\documentclass[11pt,leqno]{article}
\usepackage{ graphicx, amsfonts, amsmath,amsthm, amsxtra, verbatim, multicol, hyperref}
\usepackage[mathscr]{euscript}
\usepackage[ruled,vlined]{algorithm2e}
\begin{document}

\begin{center}
{\LARGE\bf Convolution of Picard-Fuchs equations }
\\
\vspace{.25in} {\large {\sc Hossein Movasati}} \\
Instituto de Matem\'atica Pura e Aplicada, IMPA, \\
Estrada Dona Castorina, 110,\\
22460-320, Rio de Janeiro, RJ, Brazil, \\
E-mail:
{\tt hossein@impa.br} \\
%{\tt www.impa.br/$\sim$hossein}
 {\large {\sc Stefan Reiter}} \\
Mathematisches Institut, Universit\"at Bayreuth \\
95440 Bayreuth, Germany \\
E-mail: {\tt Stefan.Reiter@uni-bayreuth.de} \\
\end{center}
%%%%%%%%%%%%%%%%%%%%%%%%%%%%%%%%%%%%%%%%%%%%%%%%%%%%%%%%%%%%%%%%%%%%%%%%%%%

%---------------------------------------------------------------------
\newtheorem{theo}{Theorem}
\newtheorem{exam}{Example}
\newtheorem{coro}{Corollary}
\newtheorem{defi}{Definition}
\newtheorem{prob}{Problem}
\newtheorem{lemm}{Lemma}
\newtheorem{prop}{Proposition}
\newtheorem{rem}{Remark}
\newtheorem{conj}{Conjecture}
%\newtheorem{prob}{Problem}
%-----------------------------------------------------------------------
\newcommand\diff[1]{\frac{d #1}{dz}} %Differential operator
\def\End{{\rm End}}              %Endomorphism group
\def\hol{{\rm Hol}}
\def\sing{{\rm Sing}}            %The set of singularities
\def\spec{{\rm Spec}}            %The spectrume
\def\cha{{\rm char}}             %Charracteristic
\def\Gal{{\rm Gal}}              %The Galois group
\def\jacob{{\rm jacob}}          %the Jacobian ideal
\newcommand\Pn[1]{\mathbb{P}^{#1}}   %Projective space of dimension #1
\def\Z{\mathbb{Z}}                   %Integer  numbers
\def\OO{{\cal O}}                       %Structural sheaf

\def\Q{\mathbb{Q}}                   %Rational  numbers
\def\C{\mathbb{C}}                   %Complex numbers
\def\as{\mathbb{U}}                  %Some affine space
\def\ring{{\sf R}}                             %A ring 
\def\R{\mathbb{R}}                   %real numbers
\def\N{\mathbb{N}}                   %natural numbers
\def\A{\mathbb{C}}                   %affine space C^n
\def\P{\mathbb{P}}                   %Unit disc
\def\uhp{{\mathbb H}}                %upper half plane
\newcommand\ep[1]{e^{\frac{2\pi i}{#1}}}% unipotent numbers
\newcommand\HH[2]{H^{#2}(#1)}        %Hodge structures
\def\Mat{{\rm Mat}}              %Matrices
\newcommand{\mat}[4]{
     \begin{pmatrix}
            #1 & #2 \\
            #3 & #4
       \end{pmatrix}
    }                                %two by two matrices
\newcommand{\matt}[2]{
     \begin{pmatrix}                 % one by two matrix
            #1   \\
            #2
       \end{pmatrix}
    }
\def\ker{{\rm ker}}              %kernel
\def\cl{{\rm cl}}                %Chern class
\def\dR{{\rm dR}}                %The subindex dR standing for de Rham
                                     %cohomology.

\def\hc{{\mathsf H}}                 %The set of Hodge cycles.
\def\Hb{{\cal H}}                    %Hodge bundle
\def\GL{{\rm GL}}                %The liner group
\def\pedo{{\cal P}}                  %Period domain
\def\PP{\tilde{\cal P}}              %the period domain/ discrete group
\def\cm {{\cal C}}                   %the set of CM Hodge structures
\def\K{{\mathbb K}}                  %Field representing R or C
\def\k{{\mathsf k}}                  %Arbitrary field
\def\F{{\cal F}}                     %Hodge filtration bundle
\def\M{{\cal M}}
\def\RR{{\cal R}}
\newcommand\Hi[1]{\mathbb{P}^{#1}_\infty}%the hyperplane at infinity
\def\pt{\mathbb{C}[t]}               %Polynomials in t
\def\W{{\cal W}}                     %weight filtration
\def\Af{{\cal A}}                    %The space of tame polynomials.
\def\gr{{\rm Gr}}                %graded pieces
\def\Im{{\rm Im}}                %image
\newcommand\SL[2]{{\rm SL}(#1, #2)}    %SL(2,Z)
\newcommand\PSL[2]{{\rm PSL}(#1, #2)}  %PSL(2,Z)
\def\Res{{\rm Res}}              %Residue

\def\L{{\cal L}}                     %The moduli of polarized lattices in a
                                     %fixed vector spaces.
\def\Aut{{\rm Aut}}              %Automorphism group of a vectorspace
\def\any{R}                          %Any subring of the field of complex
                                     %numbers.
\newcommand\ovl[1]{\overline{#1}}    %Conjugation of #1.

\def\per{{\sf  pm}}  
\def\T{{\cal T }}                    %Tangent space
\def\tr{{\sf tr}}                 %Transposition of matrices
\newcommand\mf[2]{{M}^{#1}_{#2}}     %New modular functions
\newcommand\bn[2]{\binom{#1}{#2}}    %Binomial
\def\ja{{\rm j}}                 %j of a two by two matrix
\def\Sc{\mathsf{S}}                  %Simple cycles
\newcommand\es[1]{g_{#1}}            %Eisenstein series
\newcommand\V{{\mathsf V}}           %Milnor vector space
\newcommand\Ss{{\cal O}}             %Structural sheaf
\def\rank{{\rm rank}}                %rank of a module
\def\diag{{\rm diag}}
\def\BM{{\sf H}}
\def\Fi{{S}}
\def\Re{{\rm Re}} 
\def\id{{\rm id}}
\def\a{{\bf a}}
\def\b{{\bf b}}
\def\c{{\bf c}}
\def\U{{\bf U}}
\def\ord{{\rm ord}}
\def\QQ{{\mathbb Q}}
\def\CC{{\mathbb C}}
\def\NN{{\mathbb N}}
\def\del{{\partial}}
\newcommand{\ten}{ {\otimes} }
\newcommand{\To}{ \longrightarrow }
%------------------------------------------------------------------------
%%%%%%%%%%%%%%%%%%%%%%%%%%%%%%%%%%%%%%%%%%%%%%%%%%%%%%%%%%%%%%%%%%%%%%%%%

%\input{/localhome/movasati/hossein/hodge/hodgechapters/notations.tex}
%\input{notations.tex}
\def\Ra{\mathrm{Ra}}
\def\nf{g_2}                         %the new function.
%------------------------------------------

\def\sol{{I}}

\begin{abstract}
We determine explicit generators for a cohomology group constructed from a solution of a Fuchsian linear differential equation and
describe its relation with cohomology groups  with  coefficients in a local system. % with regular singularities and over a punctured projective line.
In the parametrized case, this yields into an algorithm which computes new Fuchsian differential  equations from those depending on multi-parameters.
This generalizes the classical convolution of solutions of Fuchsian differential equations. 
%When the system comes from geometry then it corresponds to fiber product of varieties.
\end{abstract}
%{\tiny
%\tableofcontents
%}
\section{Introduction}
Explicit expressions for Picard-Fuchs equations (or Gauss-Manin connections in a general context)  attached to families of 
algebraic varieties are usually huge even if the corresponding family is simple, for examples see the first author's book \cite{ho13}. However, there are some families of algebraic varieties for which
such expressions are small enough to fit into a mathematical paper, but one is not able to
calculate them through the Dwork-Griffiths method (see for instance \cite{gr69}) or its modification in the context of Brieskorn modules, see \cite{ho13}, see also \cite{Lairez2016} for another variant of this,  
(we call this algebraic method). 
For such families, we first compute a period and then the corresponding Picard-Fuchs equation, see for instance   \cite{alenstzu} 
(we call this transcendental method). 
The main reason why computing Picard-Fuchs equations fails through the algebraic method is that in this way we produce huge polynomials and the  Groebner basis algorithm fails to work. The transcendental method is restricted to a very particular families of algebraic varieties. 

In this article, we propose a new method which uses 
the internal fibration structure of algebraic varieties in order to perform Picard-Fuchs equation computations. It involves only
solving linear equations and it is  a generalization  of  the classical convolution of solutions of Fuchsian differential
equations and Deligne's work on the cohomology with coefficients in a local system, see \cite{de70/1}. One of our main motivations
for the present work is the increasing need for explicit expressions of Picard-Fuchs equations  in 
Topological String Theory and in particular in the B-model of mirror symmetry, see for instance \cite{can91}. We are also inspired by a personal communication of the first author with Ch. Doran few years ago, in which he  expressed the importance of iterative construction of Picard-Fuchs equations in the case of Calabi-Yau manifolds. Meantime in the paper \cite{DoranMalmendier} he and Malmendier have realized this in the case of $14$ families of Calabi-Yau threefolds classified in \cite{dormor}.

Let us be given:
$$
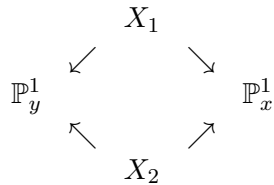
\begin{array}{lllll}
× & × & X_1 & × & ×\\
× &  \stackrel{}{\swarrow} & × & \stackrel{}{\searrow} & ×\\
\P^1_y  & × & × & × & \P^1_x \\
× & \stackrel{}{\nwarrow}& × & \stackrel{}{\nearrow}  & ×\\
× & × & X_2 & × & ×
\end{array}\  \ \     %\omega_i\in H^{n_i}(X_i/(\P^1_x\times \P^1_y)),\ \ \ i=1,2,
$$
and a global meromorphic section $\omega_i,\ i=1,2,$ of the $n_i$-th cohomology bundle
of $X_i\to \P^1_x\times \P^1_y$. Here,  $X_i,\ i=1,2,$ are two algebraic varieties over $\C$, $\P_*^1,\ *=x,y,$ is the projective line with the coordinate system $*$ and all
the arrows are morphisms of algebraic varieties. The convolution of the above data 
in the framework of Algebraic Geometry is simply the fiber product
$$
X\rightarrow \P^1_y,\ \  \
X:=\cup_{y\in \P ^1_y} X_{1,y} \times_{\P ^1_x}X_{2,y},\ \ \ 
\omega:=dx\wedge \omega_1\wedge \omega_2. %\in H^{n_1+n_2+1}(X/ \P^1_y).
$$
Here,  $X_{i,y},\ i=1,2,$ is the fiber of $X_i\to \P^1_y$ 
over the point $y\in\P^1_y$ and $\omega$ gives us a global meromorphic section of the $(n_1+n_2+1)$-the cohomology bundle of $X\to \P^1_y$.  
%In this article we deal with the following question.
%How can we compute the GMC of the fibration $X\rightarrow \P^1_s$, knowing the Gauss-Manin connection of the two parameter 
%families $X_i\to \P^1_t\times \P^1_s,\ i=1,2$? We put this question in the framework of PFE's. 
Let   $\delta_{i,x,y},\ i=1,2$ be a continuous family of cycles in the fibers of $X_i\to \P^1_x\times \P^1_y$. 
Knowing the linear differential system (Gauss-Manin connection) of $X_i\to \P^1_x\times \P^1_y$, and in particular, the Picard-Fuchs equation 
\begin{equation}\label{31july2019-1}
L_i:=p_{0,i} \del^{n_i}_x+\ldots +p_{n_i-1,i}\del_x+p_{n_i,i},\ \ p_{j,i}\in \k[x,y],\ \ i=1,2,
\end{equation}
of  the periods $\sol_i(x,y):=\int_{\delta_{i,x,y}}\omega_i$, and under certain irreducibility condition (see \ref{15july2013}) we give an algorithm for computing the Picard-Fuchs equation 
\begin{equation}\label{31july2019-2}
L:=q_{0} \del^n_y+\ldots +q_{n-1}\del_y+q_{n},\ \ q_{j}\in \k[y], 
\end{equation}
of  
\begin{equation}
\label{16july2013}
\sol(y):=\int_{\delta}\sol_1(x,y)\sol_2(x,y)dx,
\end{equation}
where $\delta$ is any closed path in the $x$-plane such that $\sol_1$ and $\sol_2$ along $\delta$ are one valued. This integral can be written as 
the integration of $\omega$ over a cycle $\tilde \delta_y \in H_{n_1+n_2+1}(X_y,\Z)$, where $X_y$ is the fiber of $X\to \P^1_y$ over $y$.

%%%%%%%%%%%%%%%%%%%%%%%%%%%%%%%%%%%%%%%%%%%%%%5
\section{Cohomology with coefficients in a local system}
\label{section1}
\def\O{{\mathcal O}}
In this section we remind some basic facts on local systems and connections with regular singularities. 
For further details, the reader is referred to  \cite{de70/1}. We fix a field $\k$ of characteristic zero and 
not necessarily algebraically closed and work over the category of algebraic varieties over $\k$. If $\k$ is a subfield of $\C$ or $\C(t)$, 
where $t$ is a multi-parameter, then for  an algebraic variety $M$ over $\k$ we use the same letter $M$ to denote 
the underlying complex variety or family of varieties; being clear in the text which we mean.

\subsection{Flat connections}
Let $M$ be a smooth variety, $E$ be a vector bundle over $M$.
 We consider a flat regular connection 
$$
\nabla: E\to \Omega^1(E).
$$
We use the same notation $E$ for both the vector bundle and the sheaf of its sections.
We have the induced maps 
$$
\nabla_i: \Omega^{i}(E)\to \Omega^{i+1}(E),\ \ \nabla_i(\omega\otimes e)=d\omega\otimes e+(-1)^i\omega\wedge\nabla(e) 
$$ 
and the integrability is by definition $\nabla_1\circ \nabla_0=0$. It implies that $\nabla_{i+1}\circ\nabla_i=0$ and so 
we have  the complex $( \Omega^{i}(E),\nabla_i)$. According to the comparison  theorem of Grothendieck, see for instance Deligne's notes \cite{de70/1}
Theorem 6.2
%page 98,
we have canonical isomorphisms
\begin{equation}
 \label{22A2013}
H^*(M, \O(E))\to \uhp^*( M^{an}, \Omega^{*}(E))\leftarrow \uhp^*( M, \Omega^{*}(E)) 
\end{equation}
of $\C$-vector spaces.
Here,  $M^{an}$ is the underlying complex variety of $M$, $\O(E)$ is the sheaf of constant sections of $E$ and $H^*(M, \O(E))$ is 
the Cech cohomology with coefficients in $\O(E)$. The first $\uhp$ is the hypercohomology in the complex context and 
the second one is the algebraic hypercohomology. Note that, $\Omega^{*}(E)$ is an algebraic sheaf and 
so its sections have poles of finite order along a compactification  of $M$. 
If $M$ is an affine variety then  $H^i(M, \Omega^{*}(E))=0$ for $i>0$ and so
\begin{equation}
\label{22aug2013}
\uhp^i( M, \Omega^{*}(E))\cong \frac{\ker \left(H^0( M, \Omega^{i}(E)) \to H^0( M, \Omega^{i+1}(E))\right)}{\Im \left(H^0( M, \Omega^{i-1}(E)) \to H^0( M, \Omega^{i}(E))\right)}
\end{equation}
see \cite{de70/1} Corollary 6.3.
%page 99. %From now on $M$ denotes an affine variety.

\subsection{Logarithmic differential forms}
Let us now consider a meromorphic connection $\nabla$ on 
$X$ with poles along a normal crossing divisor $S\subset X$, and hence, it
induces a holomorphic connection on $M:= X\backslash S$. We denote by $\Omega^1_X\langle S\rangle $ the sheaf of meromorphic differential forms in $X$ 
with only logarithmic poles along $S$. The sheaf $\Omega^p_X\langle S\rangle $ is $p$-times wedge product of $\Omega^1_X\langle S\rangle$. If $\nabla$ has only logarithmic poles along $S$, 
see \cite{de70/1} page 78, then $\nabla$ induces $\Omega^p_X\langle S\rangle (E)\to \Omega^{p+1}_X\langle S\rangle (E) $  and 
we have an isomorphism 
\begin{equation}
\label{22a2013}
\uhp^*(X, \Omega^*_X\langle S\rangle (E))\cong \uhp^*(X, \Omega^*_X(E))
\end{equation}
induced by inclusion and then restriction to $M$ provided that the residue matrix of $\nabla$ 
along the irreducible components of $S$ 
does not have eigenvalues in $\N$, see \cite{de70/1} Corollary 3.15. These conditions will appear later 
in   Theorem  \ref{11Jul2013} and Theorem \ref{11july2013}.  
 If $X$ is an affine variety then we conclude that 
in  (\ref{22aug2013}) every element is represented by a logarithmic differential, see also \cite{de70/1} Corollary 6.10. 

The hypercohomology groups (\ref{22A2013}) and (\ref{22a2013}) are finite dimensional $\k$-vector space, see \cite{de70/1} Proposition 6.10  and \cite{dim04}
Proposition 2.5.4. However, explicit bases for these cohomology groups and algorithms which compute
an element as a linear combination of  the basis, are
not the main focus of \cite{de70/1, dim04}. A regular connection may not have logarithmic poles along $S$ and one has to 
modify it in order 
to get such a property, see Manin's result in \cite{de70/1} Proposition 5.4. 
In our terminology this is the same as to write any regular
differential equation in the Okubo format, see \S\ref{okubosection}. 
All these together, leads us to that fact that the theoretical approach in \cite{de70/1} is not applicable to our main 
problem posed the Introduction.

\subsection{Relation with integrals}
Let us now consider meromorphic global sections $e_1,e_2,\ldots,e_n$ of $E$ such that
for points  $x$ in some open Zariski subset of $M$, $e_{i}(x),\ \ i=1,2,\ldots,n,$ form a basis
of $E_x$. Replacing $M$ with this Zariski subset, we can assume that this property is valid
for all $x\in M$. In this way $E$ becomes a trivial bundle. Let $e=[e_1,e_2,\ldots,e_n]$ and
$$
\nabla(e^\tr)=A\cdot e^\tr
$$
where $A$ is a $n\times n$ matrix whose entries are regular differential forms in $M$ (with poles along
the complement of $M$ in its compactification). We identify $(\Omega^i_M)^n$ with  $\Omega^{i}(E)$ through the map $\omega\mapsto \omega\cdot e^{\tr}$ and we get:
$$
\nabla_i : (\Omega^i_M)^n\to (\Omega^{i+1}_M)^n,\ \ \nabla_i \omega=d\omega+(-1)^ i\omega A 
$$
and so
$$
\uhp^i( M, \Omega^{*}(E))\cong \frac{\ker \left( H^0( M, \Omega^{i}_M)^n \to H^0( M, \Omega^{i+1}_M)^n\right )}{\Im \left
(H^0( M, \Omega^{i-1}_M)^n \to H^0( M, \Omega^{i}_M)^n\right)}
$$
Let $\check E$, $\check\nabla: \check E\to \Omega^1(\check E)$ and $\check e_i$  be the dual bundle to $E$, 
the dual connection and the dual basis, respectively. We have $\nabla \check e^\tr=-A^\tr \check e^\tr$.  
For a flat section $I$ of $\check E$ we write $I=\check e\cdot f$, where 
$f=[f_1,f_2,\ldots, f_n]^{\tr}$ and $f_i$'s are holomorphic functions in  a small open set $U$ in $M$. 
We have $0=\nabla I=(\nabla \check e) f+\check e df=\check e(df-Af)$, and so,   we get a system
$$
L:\ \ \ dY=AY
$$
with the solution $f$. 
Let us define  
$$
 H^0( M, \Omega^{i}_M)^n_f:=\{\omega\in  H^0( M, \Omega^{i}_M)^n  \mid \omega\cdot f=0\}
$$
and let $H^i_f$ be the $i$-th cohomology group of the complex $\left( H^0( M, \Omega^{i}_M)^n_f,\nabla_i\right)$. 
We also define
\begin{equation}
\label{19mar2016}
H^i_\dR(M,L):=\frac{\{ \omega=\sum_{k=1}^n f_k\omega_k\mid \omega_k\in   H^0(M, \Omega^{i}_M), \ \ d\omega=0\} } 
{\{ d(\sum_{k=1}^n f_k\omega_k)\mid \omega_k\in   H^0(M, \Omega^{i-1}_M)\} }.\ 
\end{equation}  
This depends on $f$, however, for simplicity we have not used $f$ in its notation. 
We have the exact sequence 
\begin{equation}
\label{18nov2012}
H^i_f \to \uhp^i( M, \Omega^{*}(E))\to H^i_\dR(M,L)\to H^{i+1}_f
\end{equation}
which is the part of the long exact sequence of the short
exact sequence  $0\to  H^0( M, \Omega^{i}_M)^n_f\stackrel{j}{\to}  H^0( M, \Omega^{i}_M)^n \to {\rm cokernel}(j) \to 0$.  

From now on we use the cohomology group  $H^i_\dR(M,L)$.  The advantage of this is that  we can integrate its elements.
Let $\delta$ be a topological $i$-cycle in $M$
such that the restriction of the analytic continuations of $f_k$'s to $\delta$ is one valued. 
For $\omega$ in the right hand side (\ref{19mar2016})
the integration $\int_{\delta}\omega$ is well-defined. One of our motivations in the present text is  to study this integral.
Later we will see that for $M$ the punctured line  $H^1_\dR(M,L)$ is finite dimensional $\C$-vector space and so the $\C$ vector space generated by
 $$
\int_{\delta} f_{i}\omega, \ \ i=1,2,\ldots,n,\ \ \omega\in H^0(M, \Omega^{i}_M)
$$
for a fixed $\delta$ is of finite dimension.

\section{Cohomology of linear differential equations}
\def\dR{{\rm dR}}
In this section we translate the machinery introduced in the previous section for the case of the punctured line, that is, $M$ is $\P^1$ minus a finite number of points.
We consider local systems given by Fuchsian differential equations and  we give an
explicit set of generators for the corresponding cohomology groups. For simplicity, we work with $\k\subset \C$. The case $\k\subset \C(t)$ is reduced to the previous one
by  taking $t$ as a collection of algebraically independent transcendental numbers in $\C$. 
\subsection{The case of the punctured projective line}
Let 
\begin{equation}
\label{8july}
L: Y'=AY,\ \ '=\del _x
\end{equation}
 be a Fuchsian differential system of dimension $n$ with a solution  $f=[f_1,f_2,\cdots,f_n]^\tr$. We assume that the entries of $A$ are in $\k(x)$.
 Let also $S\subset \P ^1$ be the set of singularities of  $L$. We define 
\begin{equation}
\label{12.09.2019}
H^1_\dR(\P ^1-S, L):=\frac{ \left\langle pf_idx \mid p\in \k(x),\ \ {\rm pol}(p)\subset S,\ \ i=1,2,\ldots, n\right \rangle_\k }{ \langle d(pf_i)
\mid p\in \k(x),\ \ {\rm pol}(p)\subset S,\ \ i=1,2,\ldots,n
\rangle_\k    }
\end{equation}
which is the same as in (\ref{19mar2016}) for $\k=\C$. 
For a Fuchsian differential operator  
\begin{equation}\label{L}
L:=p_0 \del^n_x+\ldots +p_{n-1}\del_x+p_n,\ \ p_i\in \k[x] 
\end{equation}
with a solution $f$, that is $Lf=0$, 
we can attach the linear differential system (\ref{8july})
with
\begin{equation}
\label{31july2019emre}
A=\begin{pmatrix}
   0&1&0&\cdots &0\\
   0&0&1&\cdots &0\\
   \vdots&\vdots &\vdots&\cdots &\vdots\\
    0&0&0&\cdots &1\\
    -\frac{p_n}{p_0}&-\frac{p_{n-1}}{p_0}&-\frac{p_{n-2}}{p_0} &\cdots & -\frac{p_1}{p_0}
  \end{pmatrix}
\end{equation}
and so
\begin{equation}
\label{12sept2019}
H^1_\dR(\P ^1-S, L):=\frac{ \langle pf^{(i)}dx \mid p\in \k(x),\ \ 
{\rm pol}(p)\subset S,\ \ i=0,1,\ldots, n-1\rangle_\k }{ \langle d(pf^{(i)})
\mid p\in \k(x),\ \ {\rm pol}(p)\subset S,\ \ i=0,1,\ldots,n-1
\rangle_\k    }.
\end{equation}
%Since $f^{(n)}$ is a linear combination of $f^{(i)},\ i=0,1,\ldots,n-1$ with coefficients in $\k(x)$ and with poles at $S$, 
%in the above equality we can use $i\in \N\cup\{0\}$. 
%%%%%%%%%%%%%%%%%%%55
\subsection{Indicial equation}
In order to study the Fuchsian differential equation \eqref{L}, $Ly=0$, near a point $t \in \P^1$ it is useful to compute its Riemann-scheme  
\cite[Section~3]{2007Beukers}. 
%For this one has to determine the roots of the indicial equation $I_t$ at $t \in\P^1$, which are called the local exponents at $t$.
%The Riemann scheme of $L$ is then given by the singular points $t \in S$ together with their local exponents. 
%Let $t_j\in\bar\k,\ j=1,2,\ldots,r$ be the set of finite singularities of $L$. 
Let 
$$
 a_{i,t}:=\lim_{x\to t} (x-t)^i \frac{p_i}{p_0}\in \k \ \ \ \ \ t\in\k, 
$$
$$
 a_{i,\infty}:=\lim_{x\to \infty} (-x)^i \frac{p_i}{p_0} \in \k. \ \ \ \ %t_j=\infty
$$
Since $L$ is Fuchsian we have
$$
\frac{p_i}{p_0}= \frac{a_{i,t}}{(x-t)^i} + R_{i,t}, \ \ \ \ord_{x=t} R_{i,t} \geq -i+1,
$$
for any finite $t\in\k$, and for any $l\geq n$ 
$$
x^l\frac{p_i}{p_0}=(-1)^i a_{i,\infty}x^{l-i}+P_{i,l,\infty}(x) + R_{i,l,\infty}(x).
$$
where $P_{i,l,\infty}(x)$ is a polynomial in $x$ of degree $<l-i$ and $R_{i,l,\infty}$ is
a sum over all finite singularities $t_j\in\bar\k $ of $L$, of polynomials in $\frac{1}{x-{t_j}}$ of degree $\leq i$.
Note that both $R_{i,l,\infty}$ and $P_{i,l,\infty}$ are defined over $\k$.  We conclude that the 
differential operator $L$ can be also written in the format
\begin{eqnarray}
\label{format}
 & &\partial_x^{(n)}+\sum_{i=1}^{n}\frac{a_{i,t}}{(x-t)^{i}}\partial_x^{(n-i)}+\sum_{i=1}^{n}R_{i,t} \partial_x^{(n-i)},\\
 \label{formatb} & &x^l\partial_x^{(n)}+\sum_{i=1}^{n} a_{i,\infty} x^{l-i}  \partial_x^{(n-i)}+\sum_{i=1}^{n}(P_{i,l,\infty}+R_{i,l,\infty }) \partial_x^{(n-i)}.
\end{eqnarray}
Furthermore the indicial equation $I_{t}$ at 
$t$ is given by
\begin{eqnarray}%\label{Indicial}
 I_{t}&=& X (X-1) \cdots (X-n+1)+a_{1,t} X (X-1) \cdots (X-n+2)+\cdots+ a_{n,t}  \label{Indicial} \\ 
 I_{\infty} &=&X (X+1) \cdots (X+n-1)+a_{1,\infty} X (X+1) \cdots (X+n-2)+\cdots+a_{n,\infty}\label{Indicialb}
\end{eqnarray}
The Riemann scheme of $L$ at a point $t\in\k\cup\{\infty\}$ is the set of the roots of $I_t$.

%%%%%%%%%%%%%%%%%%%%%%%%%%%%%%%%%%%%%%%%%%%%%%%%%%%%%%%%%55
\subsection{Explicit set of generators, $\k=\bar \k$}
We are now in a position to describe an explicit set of generators  for the cohomology group 
$H^1(\P ^1-S, L)$, where $S=\{t_1,t_2,\ldots,t_t,\infty\}$ is the set of singularities of $L$. 
\begin{theo}
\label{11Jul2013}
Let $\k$ be an algebraically closed subfield of $\C$. 
If the Fuchsian differential operator $L$ 
has no integer exponent $\geq n$ in the Riemann-scheme at a finite point and no positive integer exponent at $\infty$
then
$H^1_\dR(\P ^1-S, L)$ is generated by
\begin{equation}
\label{15aug2012}
(x-t_j)^{-1} f^{(i)} dx,\ \ j=1,2,\ldots,r, \ \ i=0,1,\ldots,n-1
\end{equation}
and so it is of dimension at most $n\cdot r$. 
\end{theo}

\begin{proof}
All the qualities below are in the cohomology group $H^1_\dR(\P ^1-S, L)$ 
 that is obviously generated by
$$
f^{(i)} x^l dx,  \frac{f^{(i)}}{(x-t_j)^l} dx,\ \ \ l \in \NN,\ \ i=0,1,\ldots, n-1,\ \ j=1,2,\ldots,r. 
$$
At first we show that 
$\frac{f^{(i)}}{(x-t_j)^l} dx,\ \ l \in \NN,\ \ i=0,1,\ldots, n-1,$ 
is in the $\k$-vector space $V$‌ generated by (\ref{15aug2012}). Since
$$
0=d(\frac{f^{(i)}}{(x-t_j)^l})=\frac{f^{(i+1)}}{(x-t_j)^l} dx+(-l) \frac{f^{(i)}}{(x-t_j)^{l+1}}dx 
$$
we get 
\begin{equation}
\label{9july2013}
\frac{f^{(i)}}{(x-t_j)^l}dx=\frac{1}{l-1}\frac{f^{(i+1)}}{(x-t_j)^{l-1}}dx=\cdots= 
\left\{ \begin{array}{cc}
   \frac{1}{(l-1)\cdots (l-n+i)}\frac{f^{(n)}}{(x-t_j)^{l-n+i}}dx & i+l \geq n+1 \\
    \frac{1}{(l-1)!}\frac{f^{(i+l-1)}}{x-t_j}dx & i+l < n+1 \\                                                                                                                                                                                         
\end{array}\right.
\end{equation}
Now, we use induction on $i+l$. For $i+l<n+1$, the claim follows from the second case in (\ref{9july2013}).
Thus by the first case in (\ref{9july2013}) we can assume $i=n$. 
We have
\begin{eqnarray*}
\frac{f^{(n)}}{(x-t_j)^l} dx &\stackrel{\eqref{format}}{=}&
-(\sum_{i=1}^n  a_{i,t_j} \frac{f^{(n-i)}}{(x-t_j)^{i+l}})dx-(\sum_{i=1}^{n }\frac{R_{i,t_j}}{(x-t_j)^l} f^{(n-i)})dx
\\
&\stackrel{\eqref{9july2013}}{=}& 
-(\sum_{i=1}^n \frac{a_{i,t_j}}{l(l+1)\cdots (l+i-1)})\frac{f^{(n)}dx}{(x-t_j)^l} -(\sum_{i=1}^{n }\frac{R_{i,t_j}}{(x-t_j)^l} f^{(n-i)})dx.
\end{eqnarray*}
In 
$
\frac{R_{i,t_j}}{(x-t_j)^l} f^{(n-i)}dx
$
there appear only terms  $\frac{f^{(n-i)}}{(x-t_j)^{l+k}} dx$ with $k\leq i-1$
and terms $\frac{f^{(n-i)}}{(x-t_{j'})^i}dx$ for $j'\neq j$.
Using the first case in (\ref{9july2013}) the former terms are by induction in $V$ and 
the latter terms by the second case in (\ref{9july2013}).
By assumption  we  get
\begin{equation*}
 1+(\sum_{i=1}^n \frac{a_{i,t_j}}{l(l+1)\cdots (l+i-1)})\stackrel{\eqref{Indicial}}{=}
    \frac{I_{t_j}(l+n-1)}{l\cdots (l+n-1)} \neq 0,\ \forall l\in\N.
\end{equation*}
Hence
$\frac{f^{(n)}}{(x-t_j)^l} dx \in V$.

Similarly we prove that 
$x^l f^{(i)} dx$ is in the vector space $V$.  If $l-i<0$ then we 
have 
\begin{equation*}
x^lf^{(i)}dx=-(l-1)x^{l-1}f^{(i-1)}dx=\cdots=0.
\end{equation*}
For $l-i\geq 0$  
we use induction on $l-i$ and we have
\begin{equation}\label{xfn}
x^lf^{(i)}dx=\frac{1}{-(l+1)}x^{l+1}f^{(i+1)}dx=\cdots=\frac{(-1)^{n-i}}{(l+1)\cdots (l+n-i)}x^{l+n-i}f^{(n)}dx
\end{equation}
and so we can assume that $i=n$. Now,  for $l\geq n$ we have by (\ref{xfn}) and (\ref{formatb})
\begin{eqnarray*}
x^l f^{(n)} dx &=& -\sum_{i=1}^n a_{i,\infty} x^{l-i} f^{(n-i)}dx- \sum_{i=1}^{n }P_{i,l, \infty} f^{(n-i)}dx- 
\sum_{i=1}^{n }R_{i,l, \infty} f^{(n-i)}dx.
\end{eqnarray*} 
The second sum is by hypothesis of induction in $V$
and the third by (\ref{9july2013}).
The first sum is by (\ref{xfn})
$$
-\left (\sum_{i=1}^n  \frac{ (-1)^ia_{i,\infty}}{(l-i+1)(l-i+2)\cdots l}\right )x^{l} f^{(n)}dx\stackrel{\eqref{Indicialb}}{=}
\left (1-\frac{I_\infty(l-n+1)}{ l\cdots (l-n+1)  }\right) x^lf^{(n)}dx.
$$
Therefore, since 
  $l-n+1$ is not an exponent at $\infty$ we get $x^lf^{(n)}dx$ is in $V$.

\end{proof}

\begin{rem}\rm
Since $x^i f^{(n)} dx =0$ in $H^1_\dR(\P ^1-S, L)$  for $i=0,\ldots,n-1,$ we obtain $n$
$\k$-linear relations between the generators $\eqref{15aug2012}$ of   $H^1_\dR(\P ^1-S, L)$.
\end{rem}

\begin{rem}\rm
 If $\infty$ is no singularity then the exponents at $\infty$ are $0,-1,\ldots,-n+1$. 
 Thus the condition that the exponent is positive is compatible with the condition that the exponents
 are $\geq n$ at the finite singularities.
\end{rem}
\begin{rem}\rm
\label{19july2013}
Without the hypothesis on indicial equations of $L$, we have to add the following
elements
\begin{eqnarray*}
\frac{f^{(n)}}{(x-t_j)^l}  dx,&& \mbox{if }  I_{t_j}(l+n-1)=0,\\
x^lf^{(n)}dx, && \mbox{if } I_{\infty}(l-n+1)=0, \;l\geq n
\end{eqnarray*}
to the set (\ref{15aug2012}) in order to get a set of generators.
\end{rem}

%%%%%%%%%%%%%%%%%%%%%%%%%%%%%%%%%%%%%%%%%%%%%%%%%%%%%
%%%%%%%%%%%%5

\subsection{Explicit set of generators, $\k\neq{\bar \k}$}
\label{3july2019}

In case $\k\not=\bar\k$ and for computational purposes we modify Theorem \ref{11Jul2013}  and reprove  it over $\k$. 
For this we proceed as follows. Let  
$$
\Delta=\prod_{i=1}^r(x-t_i),
$$
where $t_i\in\bar\k$ are the finite singular points of $L=0$ (without repetition).
Since $L$ is defined over $\k$, 
the Galois group of $\bar\k$ over $\k$ acts on $t_j$'s and so $\Delta\in\k[x]$. 
Thus we can write $L$ in the following way
\begin{equation}
\label{silveiramartins2019}
L=\sum_{i=0}^n \Delta^i \tilde p_{n-i}(x) \partial^i_x, \quad \quad  \tilde p_i \in \k[x], \;
\deg \tilde p_i\leq  i (r-1), \quad \tilde p_0=1,
\end{equation}
see for instance \cite[I. Prop. 4.2]{IKSY}.

\begin{theo}
\label{11july2013}
If $L$ 
has no integer exponent $\geq n$ in the Riemann-scheme at a finite point and no positive integer exponent at $\infty$
then
$H^1_\dR(\P ^1-S, L)$ is generated by
\begin{equation}
\label{15aug2012-D}
\frac{x^jf^{(i)}dx}{\Delta} ,\ \ j=0,1,\ldots, r-1, \  i=0,1,2,\ldots,n-1.
\end{equation}
\end{theo}
\begin{proof}
 The $\bar\k$-vector space $H^1_\dR(\P ^1-S, L)\otimes_\k \bar\k$ has a set of generators (\ref{15aug2012}) and for fixed $i\in\N$ we have
$$ 
\langle \frac{f^{(i)}dx }{x-t_j} \mid j=1,\ldots,r \rangle_{\bar\k}= \langle \frac{x^kf^{(i)}dx}{\Delta}\mid k=0,\ldots, r-1 \rangle_{\bar\k},\ \ 
i=0,1,2,\ldots,n-1.
$$
\end{proof}
In order to implement the above proof in a computer, one has
to introduce new variables $t_j$ for each singularity and
so it does not give  an effective 
algorithm which writes an element of $H^1(\P ^1-S,L )$ in terms of the generators (\ref{15aug2012-D}). 
We give a second proof which is algorithmic and does not use $\bar \k$.

By the extended Euclidean algorithm, there are  polynomials $a, b \in \k[x]$ such that
\[ 1 = a \Delta+ b \Delta'.\]
 For 
$t \in \k$ 
we define  
\begin{equation}\label{ct}
c_t:= 1+\sum_{i=1}^n\frac{\tilde p_i b^i}{t(t+1)\cdots (t+i-1)}\in\k[x].
\end{equation}
\begin{lemm}
For a fixed $t\in\k$, we have  $gcd(c_t, \Delta)=1$ if and only if  $t+n-1$ is not  an exponent 
of $L$ at finite singularities $t_j,\ j=1,2,\ldots,r$.
\end{lemm}
\begin{proof}
Let $t_j$ be a root of $\Delta$.
We have
\[ 1=b(t_j)\cdot \Delta'(t_j),\ \ \ \Delta'(t_j)= (\frac{\Delta}{x-t_j})(t_j) \]
and so
$$ 
a_{i,t_j}:=\lim_{x\to t_j} \frac{\tilde p_i(x)(x-t_j)^i}{\Delta^i}=p_i(t_j)  (\frac{1}{\Delta'(t_j)})^i =\tilde p_i(t_j)b(t_j)^i. 
$$
Thus multiplying
$c_t$ by $t(t+1)\cdots (t+n-1)$ and evaluating $x$ at $t_j$ gives the value of the indicial equation $I_{t_j}$ evaluated at $t+n-1$.
\end{proof}

\begin{proof}[Second Proof of Theorem \ref{11july2013}]
All the qualities below are in the cohomology group $H^1_\dR(\P ^1-\Delta, L)$. 
 Obviously it is generated by 
$$
 f^{(i)} x^ldx,  \frac{f^{(i)} x^k}{\Delta^l} dx,\ \ \ l\in \NN,\ \ i=0,1,\ldots, n-1,\ k=0,1,\ldots,r-1.
$$
Note that by division over $\Delta$, it is enough to consider $0\leq k<r$.
Let $V$ be the $\k$-vector space generated by (\ref{15aug2012-D}). 
At first we show that 
$\frac{f^{(j)} x^k }{\Delta^l} dx \in V$. Again for $p \in \k[x]$ we have
$$
0=d(\frac{f^{(j)} p }{\Delta^l})=\frac{(f^{(j)} p)'}{\Delta^l} dx+(-l) \frac{f^{(j)} p \Delta'}{\Delta^{l+1}}dx. 
$$
For $l \in \NN$ we get
\begin{eqnarray}
 \notag \frac{f^{(j)} x^k}{\Delta^{l+1}}dx&=&
\frac{f^{(j)} x^k(a \Delta+b \Delta')}{\Delta^{l+1}}dx\\
 \notag &=&\frac{f^{(j)} x^k a}{\Delta^{l}}dx+ \frac{f^{(j)} x^k b \Delta'}{\Delta^{l+1}}dx\\ 
 \notag &=&\frac{f^{(j)} x^k a}{\Delta^{l}}dx +\frac{1}{l} \frac{(f^{(j)} x^k b)'}{\Delta^{l}}dx  \\\label{reduction}
 &=&\frac{f^{(j)} x^k a}{\Delta^{l}}dx +\frac{1}{l} \frac{f^{(j)} (x^k b)'}{\Delta^{l}}dx+
    \frac{1}{l} \frac{f^{(j+1)} x^k b}{\Delta^{l}}dx.
\end{eqnarray}
Hence if $j<n-1$ we can reduce the pole order.
It remains to show that for $j=n-1$ we can reduce the pole order.
Let $q \in k[x]$. Then
\begin{eqnarray}
\label{qfn} \frac{q f^{(n)}}{\Delta^l} dx &=&  \frac{\Delta^n q f^{(n)}}{\Delta^{l+n}} dx \\
\notag &=& -\sum_{i=1}^n  \frac{ q \tilde p_{i} f^{(n-i)}}{\Delta^{l+i}} dx\\
\notag &\stackrel{\eqref{ct}}{=}&  -q (c_l-1)  \frac{f^{(n)}}{\Delta^l}dx+ \mbox{lower pole orders terms}
\end{eqnarray}
where the last equality follows by  (\ref{ct}) and (\ref{reduction}). 
 Since $gcd(c_l,\Delta)=1$ we have polynomials $A,B\in\k[x]$ such that $A c_l+B\Delta=1$. 
Hence the pole order of
$$\frac{x^kf^{(n)}}{\Delta^l} dx=\frac{(x^k A)c_lf^{(n)}}{\Delta^l} dx+\frac{x^k Bf^{(n)}}{\Delta^{l-1}} dx$$
can be reduced to $l-1$ using  \eqref{qfn} with $q=x^k A$. 
Reducing the pole order of $\Delta$ may yield  also terms
$x^j f^{(i)} dx$.
However by the same arguments as in the proof of Theorem \ref{11Jul2013} we have  $x^j f^{(i)} dx\in V$.
\end{proof}
\begin{rem}\rm
\label{24march2016}
In Theorem \ref{11july2013},  without the hypothesis on indicial equations of $L$, we have to add the following
finite number of elements
\begin{eqnarray*}
\frac{x^k f^{(n)}}{\Delta^l}  dx,&& \mbox{if }  0\leq k < \deg \left(gcd(\Delta,c_l)\right),\\
x^lf^{(n)}dx, && \mbox{if } I_{\infty}(l-n+1)=0, \;l\geq n
\end{eqnarray*}
to the set (\ref{15aug2012}) in order to get a set of generators.
\end{rem}

%%%%%%%%%%%%%%%%%%%%%%%%%%%%%%%%%%%%%%%%%%%%%%%%5
\subsection{Cohomologies over function fields}
\label{15july2013}
In this section we turn to the main problem posed in the Introduction, that is, how to compute the linear differential equation of
\eqref{16july2013}. Let us assume that $\k=\tilde\k(y)$, where $y$ is a variable and $\tilde\k$ is a subfield of $\C$, and so we have the 
derivation $\partial_y:\k\to\k$. 
Let 
\begin{equation}
\label{kronecker}
dY_i=A_iY_i,\ i=1,2,\ \ A_i\in\Mat_{n_i\times n_i}(\tilde \k(x,y)dx+\tilde \k(x,y)dy )
\end{equation}
be the Gauss-Manin connection of the family $X_i\to \P^1_x\times \P^1_y,\ \ i=1,2$.
We make the Kronecker product of these two systems and  obtain the system 
\begin{equation}
\label{kronecker-1}
dY=M\cdot Y,\ \
\end{equation}
where  $M=A_1\otimes I_{n_2}+I_{n_1}\otimes A_2$ and $I_{n_i}$ is the $n_i\times n_i$ identity matrix. 
A solution of \eqref{kronecker-1}  is given by $Y=Y_1\otimes Y_2$.
If we write 
$M=Adx+Bdy,\ A,B\in\Mat_{n\times n}(\tilde \k(x,y))$ then  the two dimensional  system \eqref{kronecker-1} in $x,y$ variables is 
equivalent to  $\partial_xY=AY,\ \ \partial_y Y=BY$. 
It is integrable, and hence,  $dM=-M\wedge M$ or equivalently $\partial_x B-\partial_y A=BA-AB$.

The first  entry $f$, and in general any $\k(x)$-linear combination of the entries,  of $Y$  satisfies a linear differential equation $L=0$, $L\in \k[x,\partial_x]$, with respect to the variable $x$.  
From the integrability condition we conclude that a solution of $L=0$ depends holomorphically on both $x,y$. 
We need that  $\partial_y$ induces a well-defined map  
\begin{equation}
\label{15/03/2016}
\partial_y: H^1_\dR(\P^1-S, L)\to H^1_\dR(\P^1-S, L).
\end{equation}
Note that if we use the system \eqref{8july} and the definition \eqref{12.09.2019} then \eqref{15/03/2016} is well-defined, however, for linear differential equations with the definition \eqref{12sept2019}, \eqref{15/03/2016} is not necessarily well-defined. In order to get the map \eqref{15/03/2016} we assume that the differential system $\partial_xY=AY$ is irreducible over $\k=\tilde\k(y)$, that is, there is no non-zero  $\partial_x$ invariant proper subspace of the $\k(x)$ vector space generated by the entries of $Y$. This may not be the case in general, for instance when the two systems in \eqref{kronecker} are the same. In this case we have to find the decomposition of \eqref{kronecker-1} into irreducible components. This irreduciblity condition is satisfied in may examples in which one of the systems in \eqref{kronecker}, say $i=1$, is trivial, that is $n_1=1$ and $A_1=0$, and so $M=A_2$. Therefore, we have to assume that \eqref{kronecker} for $i=2$ is irreducible. Our main examples in \S\ref{examples} are of this form.      
We conclude that the 
$\k(x)$ vector space generated by the entries of $Y$  is the same as the 
$\k(y)$ vector space generated by $f, \partial_xf, \ \partial_x^2f,\cdots$.
This implies that if we set 
$X:= \begin{bmatrix}
f, \partial_xf,\cdots, \partial_x^{n-1}f
\end{bmatrix}^\tr$ and 
write
$$
X=CY,\ \ C\in\Mat_{n\times n}(\tilde \k(x,y))
$$
then $C$ is invertible. The matrix $C$ can be computed in the following way. We have $\partial_x^mY=A_mY$ with 
\begin{equation}
\label{03aug2019-1}
A_{m+1}=\partial_x A_m+A_m\cdot A,\ \ A_1:=A
\end{equation}
and the $i$-th row of $C$ is the first row of $A_i$.  
It follows that
\begin{equation}
\label{03aug2019-2}
\partial_yX=D\cdot X,\ \ D:=\partial_y C\cdot C^{-1}+C\cdot B\cdot C^{-1}. 
\end{equation}
Let 
$$
\omega=\left[\frac{f}{\Delta}, \frac{\partial_x f}{\Delta},\cdots, \frac{\partial_x^{n-1}f}{\Delta},\cdots 
\frac{x^j f}{\Delta}, \frac{x^j \partial_x f}{\Delta},\cdots, \frac{x^j\partial_x^{n-1}f}{\Delta},\cdots
\right ]^\tr
$$ 
be the $nr\times 1$ matrix containing the elements  (\ref{15aug2012-D}).
We write \eqref{15/03/2016} in $\omega$:
\begin{equation}
\label{2ag2019}
\partial_y\omega=E\cdot \omega.
\end{equation}
The matrix $E$ can be computed in the following way. We have
\begin{equation}
\label{paissandu186}
\partial_y\left(  \frac{x^j \partial_x^i f}{\Delta}\right)=
\frac{x^j \partial_y \partial_x^i f}{\Delta}- \frac{x^j\partial_y\Delta\cdot  \partial_x^i f}{\Delta^2}
\end{equation}
The first term can be written in the basis $\omega$ using \eqref{03aug2019-2}. For the scond term we have to use
pole order reduction as in the proof of Theorem \ref{11july2013}.
 The differential system
\begin{equation}
\label{finalsystem}
\del_y W=E\cdot W
\end{equation}
is satisfied by $W=\int \omega$, where the integration takes place over 
a fixed closed path in the $x$-domain such that the entries of $\omega$ are one valued.
Let $\sol_i,\ i=1,2$  be as in the Introduction. By convention $\sol_i$ is the first entry of $Y_i$, and hence, $\sol_1\sol_2$ is the first entry of $Y$ is \eqref{kronecker-1}. In order to compute the Picard-Fuchs equation of  the integral \eqref{16july2013}, we have to write $f=\sol_1\sol_2$ as a $\tilde\k(y)$-linear combination of the entries of $\omega$. Then we compute the Picard-Fuchs equation of the same linear combination of the entries of $W$ using the system \eqref{finalsystem}.  

\begin{rem}\rm
The output linear differential equation $L$ of our algorithm is not necessarily the differential equation of minimal order annihilating $W$. The fact that (\ref{15aug2012-D}) might not form a basis of the cohomology $H^1_\dR(\P^1-S, L)$ might result in this kind of phenomena. As far as the authors are aware,  this defect is also present in almost all algorithms in the literature  for computing Picard-Fuchs equations. One has to use other algorithms in order to decompose $L$ into irreducible factors and check the minimality.  	
\end{rem}	

The above process generalizes the classical convolution. In  a geometric context this is as follows.
Let  $X_i:=\tilde X_i\times \P^1_y\to \P^1_y,\ i=1,2$ be projections on the second coordinate, $X_1\to \P^1_x$ be a morphism which does not depend on the second coordinate
and $X_2\to \P^1_x$ be of the form $y- M$, where $M: X_2\to \P^1_x$ does not depend on the  second coordinate
 and by abuse of notation we have used $y$ as the projection map on the second coordinate.
 We can see easily that $\sol_1(x,y)=\sol_1(x)$ and $\sol_2(x,y)=\sol_2(y-x)$.
The integral (\ref{16july2013}) in this case is the classical  convolution.
The details of the classical convolution will be explained  in the next section.

\subsection{The algorithm}

% \begin{algorithm}[H]
% \label{PrimCycAlg}
% \KwData{ A matrix  $[p_0,p_1,\cdots,p_n]$ with entries in $\k(x)$ representing the PF equations \eqref{L}.}
% \KwResult{The Picard-Fuchs equation of \eqref{16july2013}. }
% 	\Begin{
% 	     
% 		Compute the Kronecker product $A$ of $A_1$ and $A_2$ as in \eqref{kronecker-1}\;
% 		Compute the linear diferential equation $L$ of the the first entry of $f$ in $df=Af$. This can be done for instance by 
% 		{\tt sysdif} from {\tt foliation.lib} \; 
% 		\Return $L$\;
% 
% 	}
% \caption{Writting in a basis }
% \end{algorithm}

\begin{algorithm}[H]
\label{ConvAlg}
\KwData{ Two matrices $A_1$ and $A_2$  with entries in $\tilde \k(x,y)dx+\tilde \k(x,y)dy$ representing the systems \eqref{kronecker}.
}
\KwResult{The matrix $[q_{0},q_{1},\cdots,q_{n}]$ with entries in $\tilde \k[y]$  representing the Picard-Fuchs equation
\eqref{31july2019-2}. }
	\Begin{

		 %Two matrices $[p_{0,i},p_{1,i},\cdots,p_{n_i,i}], \ \ i=1,2$ with entries in $\tilde\k[x,y]$ 
		%representing PF equations in \eqref{31july2019-1}, 
		%or	\; 			
		%If the first data is given transform it to matrices $A_1$ and $A_2$ as in \eqref{31july2019emre} \;
	Compute the Kronecker product $M:=Adx+Bdy$ of $A_1$ and $A_2$ as in \eqref{kronecker-1}\;
		Compute the linear differential equation $L$, represented by $[p_0,p_1,\ldots,p_m]$ of  the first entry of $f$ of $Y$ in $\partial_xY=AY$. 
		This can be done for instance by 
		{\tt sysdif} from {\tt foliation.lib}, see \cite{ho13} \;
		Compute $A_m$'s through the recursion \eqref{03aug2019-1} and then $C$ whose $i$-th row 
		is the first row of  $A_i$ \;
		Compute $D:=\partial_y C\cdot C^{-1}+C\cdot B\cdot C^{-1}$ in \eqref{03aug2019-2}\; 
		Apply the algorithm in Theorem \ref{11july2013} and write 
		$\frac{x^j\partial_y\Delta\cdot  \partial_x^i f}{\Delta^2}$ (and $f$ itself) in terms of the generators  $\omega$.  This involves computing $\tilde p_i$'s in \eqref{silveiramartins2019}. In the case of $f$ let us denotes the coefficients by $Q$  \;
		Compute $E$ in \eqref{2ag2019} using \eqref{paissandu186}. \;
		Compute the Picard-Fuchs equation $L$, represented by the matrix $[q_{0},q_{1},\cdots,q_{n}]$,  of the linear combination of the  entries of the system \eqref{finalsystem} with coefficients coming from $Q$\; 
		\Return $L$\;
	}
\caption{Computation of the convolution of two PF equations.}
\end{algorithm}

\section{The classical convolution}
In this section we remind the classical 
convolution of two solutions of Fuchsian linear differential 
equations and argue that
the material in \S\ref{section1} is a generalization of this concept. 
%The only difference is that our approach works for Fuchsian systems in an arbitrary format and it gives an algorithm for computing the convolution, whereas for the classical convolution we have a closed formula and it is 
%for Fuchsian differential equations in a particular format. 

\subsection{Okubo system}
\label{okubosection}
 A linear differential system of the format
\begin{equation}
\label{okubos}
(xI_{n}-T)Y'=A Y,\; 
 \end{equation}
$$
T=\diag(t_1I_{n_1},\ldots,t_rI_{n_r}),\; \sum n_i=n,\;
 \; T, A\in \Mat_{n\times n}(\C)
$$
 is called an Okubo system (in normal form) and it is a useful format for doing computations, such as convolution.
 For a Fuchsian system
 \[ D_{a}:  Y'= \sum_{i=1}^r \frac{a_i}{x-t_i}Y,\; a_i \in \Mat_{n\times n}(\C),\ \ a=(a_1,a_2,\cdots,a_r)
 \]
 we introduce the following special Okubo system:
\begin{equation}
\label{18july13}
D_{c_\mu(a)} : (xI_{nr}-T)X'= c_\mu(a) X,\;% T=\diag(t_1E_{n},\ldots,t_rE_{n}),\;
 c_\mu(a):=\left(\begin{array}{ccc}
                        a_1 & \dots & a_r \\
                         \vdots & \vdots & \vdots \\
                         a_1 & \dots & a_r 
                 \end{array} 
\right)+\mu I_{nr},\; \mu \in \CC. 
\end{equation}
where  $n_i=n,\;i=1,\ldots,r$. Via the following procedure 
we see that any Fuchsian system
is a factor system of an Okubo system.
 Let $f(x)$ be a solution of $D_a.$ Then
\begin{eqnarray}\label{tildef}
\tilde{f}(x) &:=& \left(\begin{array}{c}f(x)(x-t_1)^{-1}\\
\vdots\\
f(x)(x-t_r)^{-1}\end{array}\right)
\end{eqnarray}
satisfies the  Okubo system $D_{c_{-1}}(a)$. % see for instance \cite{DetRei}.\marginpar{\tiny reference}
If $\k$ is not not algebraically closed we work with the following equivalent Okubo system defined over $\k$.
Let $L$ be  a Fuchsian system of dimension $n$
 \begin{eqnarray*} L:&& Y' =\sum_{i=1}^{r} \frac{x^{i-1}}{\Delta} \tilde{a}_i Y \end{eqnarray*}       
 with $\tilde{a}_i \in \Mat_{n\times n}(\k)$  and
 $\Delta=\sum_{i=0}^{r} b_{r-i} x^i\in \k[x], b_0=1$.
 If $f$ is a solution of $L$ then
 \[ (f/\Delta,xf/\Delta,\ldots,x^{r-1}f/\Delta)^{\tr}\]
 satisfies the Okubo system
 \[ (xI_{rn}-\tilde{T})Y'=\tilde{A} Y,\]
 where
 \begin{eqnarray*} \tilde{T}=\left( \begin{array}{ccccccc}
                        0 & I_n & 0  \cdots& 0\\
                            0 & 0 &\ddots &\vdots  \\
                          0 & \cdots & 0 &I_n \\
                            -b_r I_n&\cdots &-b_2 I_n &-b_1 I_n
                     \end{array} \right),&&                     
  \tilde{A}=   \left( \begin{array}{ccccccc}
                       0 & \cdots & 0\\
                       \vdots & \cdots& \vdots \\
                         0 & \cdots & 0\\
                         \tilde{a}_1& \cdots &\tilde{a}_{r}
                      \end{array}\right)-I_{rn}.
 \end{eqnarray*}

\subsection{Cohomology of Okubo system}
 For the matrix $A$ in (\ref{okubos}) we consider its submatrices $A=[A_{ij}]$ according to the
 partition $n=n_1+n_2+\cdots+n_r$ and in particular its   $n_i\times n_i$ 
 submatrices $A_{ii}$ lying in the diagonal of $A$.  

 \begin{theo}
\label{theo-okubo}
Let $L$ be  the Okubo system   (\ref{okubos}) with solution $f$.
If
\begin{equation}
 \label{10Aug2012}
\det(A+mI_{n\times n})\not =0, \ \det(A_{ii}-mI_{n_i\times n_i})\not =0,\ \ \forall m\in\N
\end{equation}
then
$H^1_\dR(\P ^1-S, L)$ is generated by
\begin{equation}
\label{10aug2012}
(x-t_i)^{-1}f_jdx,\ \ i=1,2,\ldots,r, \ \ j=1,2,\ldots,n
\end{equation}
and so it is of dimension at most $n\cdot r$. 
\end{theo}
\begin{proof}
 In $H^1_\dR(\P ^1-S, L)$ and for $m\not=-1$ we have
\begin{eqnarray*}
(x-t_i)^mf_jdx &=& -(m+1)^{-1}(x-t_i)^{m+1}f_j'dx\\ 
 &=& -(m+1)^{-1}(x-t_i)^{m+1}(x-t_j)^{-1}\sum_{k=1}^n a_{jk}f_kdx,
 \end{eqnarray*}
 where $A=(a_{jk})$.
If $m$ is negative and $t_i\neq t_j$ then we have reduced the pole order. 
If $m$ is negative and $t_i=t_j$, in order to reduce the pole order 
 we need $A_{ii}+(m+1)I_{n_i\times n_i}$ to be  invertible and
 if $m$ is positive or zero we need that $A+(m+1)I_{n\times n}$ to be invertible.
\end{proof}

\begin{rem}\rm
For a linear differential equation of an entry of the Okubo system,  the conditions (\ref{10Aug2012})
imply the conditions in Theorem \ref{11july2013}.  Without these conditions
a similar observation as in Remark \ref{19july2013} is valid.
 \end{rem}

\subsection{Convolution of Okubo systems}

Given two solutions of two Okubo systems we can easily determine the Okubo system that is satisfied by their convolution.  
\begin{theo}\label{Poch} 
Let  $f_i(x)$ be a solution   of the Okubo system $(xI_{n_i}-T_i) Y_i'= A_iY_i,\ i=1,2,\  A_i\in\Mat_{n_i\times n_i}(\C)$.
 Then 
 $$
 \int_{} f_1(x) \ten f_2(y-x)dx,
 $$
 where 
 the integration is over a path in the $x\in\C$ plane  such that the integrand  is one valued, 
is a solution matrix for the Okubo system
\begin{equation}
\label{19july13}
(yI_{n_1n_2}-T_1 \ten I_{n_2}-I_{n_1}\ten T_2)Y'= (A_1 \ten I_{n_2} + I_{n_1} \ten A_2 +I_{n_1n_2}) Y. 
\end{equation}
 \end{theo}
 The proof of the above theorem is similar to \cite[Lemma~4.2]{2007-DettweilerReiter}.
Note that  the system (\ref{19july13})
 has singularities at $t^1_{i}+t^2_{j},\; i=1,\ldots,r_1,\; j=1,\ldots,r_2,$ and possibly at infinity. 
%%%%%%%%%%%%%%%%
\subsection{Convolution of Fuchsian systems}
In the case of Fuchsian systems we proceed as follows.
 Let 
 \[ 
 D_{a^i}:  Y'= \sum_{j=1}^{r_i} \frac{a_{j}^i}{x-t_{j}^i}Y, \ \ \ a_{j}^i \in 
 \Mat_{n_i\times n_i}(\C),\ i=1,2
 \]
  be two Fuchsian systems with solutions $f_1, f_2$ resp..
 Then the Okubo system
 \[   
  (yI_{n_1r_1n_2r_2}-(T_1\ten I_{n_2r_2}+I_{n_1r_1}\ten T_2))Y'= (c_0(a^1)\ten I_{n_2r_2} + I_{n_1r_1} \ten c_0(a^2)-I_{n_1r_1n_2r_2})Y 
  \]
  has 
  $ \int \tilde{f}_1(x)\otimes \tilde{f}_2(y-x) dx
  $
  as solution with $\tilde{f}_1, \tilde{f}_2$ as in \eqref{tildef}.

\section{Examples}
\label{examples}
In this section we discuss some examples of families of algebraic varieties 
whose Picard-Fuchs equation can be computed through  the methods introduced
in this article.  
We consider the case in which we have only the family  $X_1$
(take $X_2$ the product of some variety with $\P^1_x\times\P^1_t$).
In this case we want to use the Gauss-Manin connection of the two parameter family $X_1\to 
\P^1_x\times\P^1_t$ and integrate it over the variable $x$ and obtain the Picard-Fuchs equation of the one parameter family $X_1\to \P^1_t$ obtained by the composition 
$X_1\to 
\P^1_x\times\P^1_t\to \P^1_t $, where the second map is the projection. A well-known example for this situation is the 
Legendre family of elliptic curve that is given by
$$
y^2=x(x-1)(x-t)
$$
where $t$ is a parameter. We can compute compute the Picard-Fuchs equation of $\int \frac{dx}{y}$ either by direct methods or by the methods introduced
in this article:
$$
I+(8t-4)I'+(4t^2-4t)I''=0.
$$
In the second case we consider $x$ and $t$ as a parameter and so we get a two parameter family of zero dimensional varieties with two points. 
For what follows, the polynomial $\Delta$ need not to be monic in the $x$ variable. 
Let us consider  the rank 19 family of K3 surfaces given in the affine coordinates $(x,y,w)$ by 
the equation $P=0$, where
$$
P:=y^2w-4x^3+3axw^2+bw^3+cxw-(1/2)(dw^2+w^4)=0
$$
$$a=(16 + t)(256 + t),\ \ b= (-512 + t)(-8 + t)(64 + t), \  c=0, \  d=2985984t^3
$$
and $t$ is a parameter,  see \cite[Section 6.7]{GMCD-K3}. 
Here, we would like to compute the Picard-Fuchs equation of the holomorphic $2$-form given by $\omega=\frac{dx\wedge dy\wedge dw}{dP}$. 
The generic member of the family has two isolated singularities and so one cannot apply the Griffiths-Dwork method or its modification using Brieskorn modules. 
In order to apply the methods introduced in this article, 
we look $P=0$ as a two parameter family  of elliptic curves depending on $(t,w)$. 
In this case we know the explicit expression of Gauss-Manin connection, see for instance \cite[Section 6]{GMCD-K3}. 
Using this we can compute the following differential 
equations for the elliptic integral $f(t,w):=\int \frac{dx\wedge dy}{dP}$
\begin{eqnarray*}
L &:=& A_1f+A_2\partial_w f+A_3 \partial_w^2f =0  \\
 & & B_1f+B_2\partial_t f+B_3 \partial_t^2 f =0 
\end{eqnarray*}
where $A_i,B_i$'s are explicit polynomials in $w,t$ with rational coefficients:
{\tiny
\begin{eqnarray*}
 A_1 &=&  (1283918464548864t^{9}w-133116666404426219520t^{9}+1486016741376t^{8}w^{2}-585466819834281984t^{8}w+
 \\ & & 72814820327424t^{7}w^{2}-
37469876469394046976t^{7}w+ 1719926784t^{6}w^{3}-2077451404443648t^{6}w^{2} \\ & & 
+336571521970697404416t^{6}w-784286613504t^{5}w^{3}
+298249504061128704t^{5}w^{2}-50194343264256t^{4}w^{3}+ \\ & &24931223849681289216t^{4}w^{2}-144t^{3}w^{5}-474771456t^{3}w^{4}+ 
450868486864896t^{3}w^{3}+65664t^{2}w^{5}+\\ & &
4202496tw^{5}+77w^{6}-37748736w^{5})\\
 A_2 &=&  (3851755393646592t^{9}w^{2}-1916879996223737561088t^{9}w+2972033482752t^{8}w^{3}-1756400459502845952t^{8}w^{2}+
 \\ & & 
 145629640654848t^{7}w^{3}-
 112409629408182140928t^{7}w^{2}+1719926784t^{6}w^{4}-3138467357786112t^{6}w^{3}+\\ & &
 1009714565912092213248t^{6}w^{2}-497664t^{5}w^{5}- 784286613504t^{5}w^{4}+596499008122257408t^{5}w^{3}
 -24385536t^{4}w^{5}- \\ & & 50194343264256t^{4}w^{4}+49862447699362578432t^{4}w^{3}-432t^{3}w^{6}-  119439360t^{3}w^{5}+450868486864896t^{3}w^{4}+
 \\ & &
 196992t^{2}w^{6}- 
 99883155456t^{2}w^{5}+12607488tw^{6}-8349416423424tw^{5}+
144w^{7}-113246208w^{6})\\
 A_3 &=&  36 w^2( -w^{2}+2985984t^{3}) 
\left( -w^{4}+(4t^{3}-1824t^{2}-116736t+1048576)\cdot w^{3}+(6912t^{5}+338688t^{4}-7299072t^{3}+ \right. \\ & &
 1387266048t^{2}+115964116992t)\cdot w^{2}+ 
(11943936t^{6}-5446434816t^{5}-348571828224t^{4}+3131031158784t^{3})\cdot w \\ & &
\left. 
-8916100448256t^{6}\right )\\
 B_1 &=& -\frac{1}{36}A_3\cdot ( -w^{2}+2985984t^{3})^{-1}\cdot  \\ & &  
\left( 2985984t^{4}w^{2}-3456t^{3}w^{3}+2842656768t^{3}w^{2}-34560t^{2}w^{3}+73383542784t^{2}w^{2}-7tw^{4}+2211840tw^{3}-952w^{4}+\right.\\ & & 
\left.  905969664w^{3}\right)   \\
 B_2 &=& -13824t^{3}w^{2}+9746251776t^{3}w-138240t^{2}w^{2}+293534171136t^{2}w-24tw^{3}+8847360tw^{2}-3264w^{3}+3623878656w^{2} \\ 
B_3 &=&  %23887872t^{5}+6497501184t^{4}+97844723712t^{3}-8t^{2}w^{2}-2176tw^{2}-32768w^{2} 
8(t+256)(t+16)(2985984t^3-w^2).
\\
\end{eqnarray*}
} 
We use the
second equality in order to compute the action of $\partial_t$ on the cohomology group constructed from $L$. 
This data is enough to compute the Picard-Fuchs equation of the integral $g(t)=\int f(t,w)dw$ using the techniques introduced in this article. 
Note that we have to use Theorem \ref{11july2013} together with Remark \ref{24march2016} because the differential equation $L$ has the apparent
singularities $-w^{2}+2985984t^{3}$. The end result has a factor
\begin{equation}
\label{5aug2013}
\tilde L:= 1+ (26t+512) \partial_t+ (36t^2+1536t)\partial_t^2+(8t^3+512t^2)\partial_t^3  % +0I'''=0
\end{equation}
where $\tilde Lg=0$. This differential operator is obtained by direct computations as in \cite[Chapters 10,12]{ho13}. 
The function $g$ can be written as  the period of $\frac{dx\wedge dy\wedge dw}{dP}$ over two dimensional cycles living in 
the $K3$-surface, for further details see \cite{GMCD-K3}. Note that the generic fiber of $P=0$ is singular and in order to apply the algorithms in \cite{ho13}, in \cite{GMCD-K3} we have used a new parameter $s$ and computed the Gauss-Manin connection of the five parameter family of K3 surfaces $P-s=0$. For this we had to run our computer for a few hours and  the outcome data of the Gauss-Manin connection is more than $4$ mega bytes. Despite the fact we have not computed \eqref{5aug2013} by methods introduced in this paper, we believe that it is faster as it computes Picard-Fuchs equation by increasing the dimension one by one, and the available algorithms for higher dimensional families, despite being correct, do not work in practice.

\def\cprime{$'$} \def\cprime{$'$} \def\cprime{$'$} \def\cprime{$'$}

%\bibliography{../Biblio/biblio}
%\bibliographystyle{alpha}

\end{document}